\documentclass[12pt]{amsart}%
\usepackage{graphicx,amscd, color,amsmath,amsfonts,amssymb,geometry,amssymb}
\usepackage{amsmath}
\usepackage{amsfonts}
\usepackage{amssymb}
\usepackage{graphicx}%
\providecommand{\U}[1]{\protect\rule{.1in}{.1in}}
\providecommand{\U}[1]{\protect\rule{.1in}{.1in}}
\newtheorem{theorem}{Theorem}[section]

\numberwithin{equation}{section}

\begin{document}
\title[A new proof that the power $\frac{2m}{m+1}$ is sharp]{A simple proof that the power $\frac{2m}{m+1}$ in the Bohnenblust--Hille
inequalities is sharp}
\author[Daniel Nu\~{n}ez-Alarc\'{o}n and Daniel Pellegrino]{Daniel Nu\~{n}ez-Alarc\'{o}n and Daniel Pellegrino}
\address{Departamento de Matem\'{a}tica \\
Universidade Federal da Para\'{\i}ba \\
58.051-900 - Jo\~{a}o Pessoa, Brazil.}
\email{pellegrino@pq.cnpq.br and dmpellegrino@gmail.com}
\thanks{}
\keywords{Bohnenblust--Hille inequalities}

\begin{abstract}
The power $\frac{2m}{m+1}$ in the polynomial (and multilinear)
Bohnenblust--Hille inequality is optimal. This result is well-known but its
proof highly nontrivial. In this note we present a quite simple proof of this fact.

\end{abstract}
\maketitle

\section{Introduction}

The polynomial and multilinear Bohnenblust--Hille inequalities were proved by
H.F. Bohnenblust and E. Hille in 1931 and play a crucial role in different
fields as Fourier and Harmonic Analysis and Quantum Information Theory (see
\cite{annals, defant, diniz2}).

The polynomial Bohnenblust--Hille inequality proves the existence of a
positive function $C:\mathbb{N}\rightarrow\lbrack1,\infty)$ such that for
every $m$-homogeneous polynomial $P$ on $\mathbb{C}^{N}$, the $\ell_{\frac
{2m}{m+1}}$-norm of the set of coefficients of $P$ is bounded above by $C_{m}$
times the supremum norm of $P$ on the unit polydisc. This result has important
striking applications in different contexts (see \cite{annals}). The
multilinear version of the Bohnenblust--Hille inequality asserts that for
every positive integer $m\geq1$ there exists a sequence of positive scalars
$\left(  C_{m}\right)  _{m=1}^{\infty}$ in $[1,\infty)$ such that
\[
\left(  \sum\limits_{i_{1},\ldots,i_{m}=1}^{N}\left\vert T(e_{i_{^{1}}}%
,\ldots,e_{i_{m}})\right\vert ^{\frac{2m}{m+1}}\right)  ^{\frac{m+1}{2m}}\leq
C_{m}\sup_{z_{1},\ldots,z_{m}\in\mathbb{D}^{N}}\left\vert T(z_{1},\ldots
,z_{m})\right\vert
\]
for all $m$-linear forms $T:\mathbb{C}^{N}\times\cdots\times\mathbb{C}%
^{N}\rightarrow\mathbb{C}$ and every positive integer $N$, where $\left(
e_{i}\right)  _{i=1}^{N}$ denotes the canonical basis of $\mathbb{C}^{N}$ and
$\mathbb{D}^{N}$ represents the open unit polydisk in $\mathbb{C}^{N}$.

The original proof (\cite{bh}) that the power $\frac{2m}{m+1}$ is optimal is
quite puzzling. According to Defant \textit{et al (\cite[page 486]{annals})},
Bohnenblust and Hille \textquotedblleft showed, through a highly nontrivial
argument, that the exponent $\frac{2m}{m+1}$ cannot be
improved\textquotedblright\ or according to Defant and Schwarting \cite[page
90]{DS}, Bohnenblust and Hille showed \textquotedblleft with a sophisticated
argument that the exponent $\frac{2m}{m+1}$ is optimal\textquotedblright. In
\cite{Blei} there is an alternative proof for the case of multilinear
mappings, but the arguments are also nontrivial, involving $p$-Sidon sets and
sub-Gaussian systems. The main goal of this note is to present a quite
elementary proof (which solves simultaneously the cases of polynomials and
multilinear mappings) of the optimality of $\frac{2m}{m+1}$.


\section{The new proof of the sharpness of $\frac{2m}{m+1}$}

We will show that the optimality of the power $\frac{2m}{m+1}$ is a
straightforward consequence of the following famous result known as
Kahane-Salem-Zygmund inequality~(see \cite[Theorem 4, Chapter 6]{Kah} or
\cite[page 21]{Bay}):


\begin{theorem}
[Kahane-Salem-Zygmund inequality]Let $m,n$ be positive integers. Then there
are signs $\varepsilon_{\alpha}=\pm1$ so that the $m$-homogeneous polynomial
\[
P_{m,n}:\ell_{\infty}^{n}\rightarrow\mathbb{C}%
\]
given by%
\[
P_{m,n}=%
{\textstyle\sum_{\left\vert \alpha\right\vert =m}}
\varepsilon_{\alpha}z^{\alpha}%
\]
satisfies%
\[
\left\Vert P_{m,n}\right\Vert \leq Cn^{\left(  m+1\right)  /2}\sqrt{\log m}%
\]
where $C$ is an universal constant (it does not depend on $n$ or $m$)$.$
\end{theorem}

\begin{theorem}
The power $\frac{2m}{m+1}$ in the Bohnenblust--Hille inequalities is sharp.
\end{theorem}

\begin{proof}
Let $m\geq2$ be a fixed positive integer. For each $n$, let $P_{m,n}%
:\ell_{\infty}^{n}\rightarrow\mathbb{C}$ be the $m$-homogeneous polynomial
satisfying the Kahane-Salem-Zygmund inequality. For our goals it suffices to
deal with the case $n>m.$

Let $q<\frac{2m}{m+1}$. Then a simple combinatorial calculation shows that%
\[
\left(
{\textstyle\sum_{\left\vert \alpha\right\vert =m}}
\left\vert \varepsilon_{\alpha}\right\vert ^{q}\right)  ^{1/q}=\left(
p(n)+\frac{1}{m!}%
{\textstyle\prod\limits_{k=0}^{m-1}}
(n-k)\right)  ^{\frac{1}{q}},
\]
where $p\left(  n\right)  >0$ is a polynomial of degree $m-1.$ If the
polynomial Bohnenblust--Hille inequality was true with the power $q$, then
there would exist a constant $C_{m,q}>0$ so that%
\[
C_{m,q}C\geq\frac{1}{n^{\left(  m+1\right)  /2}\sqrt{\log m}}\left(
p(n)+\frac{1}{m!}%
{\textstyle\prod\limits_{k=0}^{m-1}}
(n-k)\right)  ^{1/q}%
\]
for all $n$. If we raise both sides to the power of $q$ and make
$n\rightarrow\infty$ we obtain
\[
\left(  C_{m,q}C\right)  ^{q}\geq\lim_{n\rightarrow\infty}\left(  \frac
{r(n)}{m!n^{q\left(  m+1\right)  /2}\left(  \sqrt{\log m}\right)  ^{q}}%
+\frac{p(n)}{n^{q\left(  m+1\right)  /2}\left(  \sqrt{\log m}\right)  ^{q}%
}\right)  ,
\]
with%
\[
r(n)=%
{\textstyle\prod\limits_{k=0}^{m-1}}
(n-k).
\]
Since
\[
\deg r=m>\frac{q(m+1)}{2}%
\]
we have%
\[
\lim_{n\rightarrow\infty}\left(  \frac{r(n)}{m!n^{q\left(  m+1\right)
/2}\left(  \sqrt{\log m}\right)  ^{q}}+\frac{p(n)}{n^{q\left(  m+1\right)
/2}\left(  \sqrt{\log m}\right)  ^{q}}\right)  =\infty,
\]
a contradiction. Since the multilinear Bohnenblust--Hille inequality (with a
power $q$) implies the polynomial Bohnenblust--Hille inequality with the same
power, we conclude that $\frac{2m}{m+1}$ is also sharp in the multilinear case.
\end{proof}


\bigskip

\begin{thebibliography}{9}                                                                                                %


\bibitem {Bay}F. Bayart, Maximum modulus of random polynomials, Quart. J. Math
\textbf{63} (2012), 21--39.

\bibitem {Blei}R. Blei, Analysis in integer and fractional dimensions,
Cambridge Studies in Advances Mathematics, 2001.

\bibitem {bh}H.F. Bohnenblust and E. Hille, On the absolute convergence of
Dirichlet series, Ann. of Math. \textbf{32} (1931), 600-622.

\bibitem {annals}A. Defant, L. Frerick, J. Ortega-Cerd\'{a}, M. Ouna\"{\i}es
and K. Seip, The polynomial Bohnenblust--Hille inequality is hypercontractive,
Ann. of Math. (2) \textbf{174} (2011), 485--497.

\bibitem {defant}A. Defant. D. Popa and U. Schwarting, Coordinatewise multiple
summing operators in Banach spaces, J. Funct. Anal. \textbf{259} (2010), 220--242.

\bibitem {DS}A. Defant, U. Schwarting, Bohr's radii and strips -- a
microscopic and a macroscopic view, Note Mat. \textbf{31} (2011), 87--101.

\bibitem {diniz2}D. Diniz, G.A. Mu\~{n}oz-Fern\'{a}ndez, D. Pellegrino and
J.B. Seoane-Sep\'{u}lveda, Lower bounds for the constants in the
Bohnenblust--Hille inequality: the case of real scalars, Proc. Amer. Math.
Soc., in press.

\bibitem {Kah}J.-P. Kahane, Some Random Series of Functions, Cambridge Studies
in Advanced Mathematics 5, Cambridge University Press, Cambridge, 1993.
\end{thebibliography}
\end{document}